\documentclass{article}
\usepackage{etex}
\usepackage[utf8]{inputenc}
\usepackage{algorithmicx}
\usepackage{algorithm}
\usepackage{algpseudocode}
\usepackage{amsthm, amsfonts}
\usepackage{a4wide}
\usepackage{xcolor}[dvipsnames]
\usepackage{amsmath}
\usepackage{tikz}
\usepackage{amssymb}

\usepackage{caption}
\usepackage[hypertexnames=false]{hyperref}
\usepackage{cleveref}

\algrenewcommand\algorithmicindent{3.0em}%

\newtheorem{theorem}{Theorem}[]
\newtheorem{lemma}[theorem]{Lemma}

\newtheorem{definition}[theorem]{Definition}

\newtheorem{observation}[theorem]{Observation}
\newtheorem{remark}[theorem]{Remark}

\date{}
\makeatletter
\def\@fnsymbol#1{\ensuremath{\ifcase#1\or \dagger\or \ddagger\or
   *\or \mathparagraph\or \|\or \mathsection\or \dagger\dagger
   \or \ddagger\ddagger \else\@ctrerr\fi}}
\makeatother

\begin{document}

\title{Partitioning into degenerate graphs in linear time}
\author{
Timothée Corsini\thanks{Univ. Bordeaux, LaBRI, CNRS, Bordeaux INP, Talence, France.} \ \
Quentin Deschamps\thanks{Univ. Lyon, Université Lyon 1, LIRIS UMR CNRS 5205, F-69621, Lyon, France.} \ \
Carl Feghali\thanks{Univ. Lyon, EnsL, UCBL, CNRS, LIP, F-69342, Lyon Cedex 07, France.}\\
Daniel Gon\c{c}alves\thanks{LIRMM, Univ Montpellier, CNRS, Montpellier, France.} \ \
Hélène Langlois\thanks{CERMICS, École des Ponts ParisTech, 77455 Marne-la-Vallée, France and LIGM, Univ. Gustave Eiffel, 77454 Marne-la-Vallée, France.} \ \
Alexandre Talon\thanks{Univ. Grenoble Alpes, CNRS, Grenoble INP, G-SCOP, 38000 Grenoble, France and CNRS - Sorbonne Université.}
}

\maketitle

\begin{abstract}
Let $G$ be a connected graph with maximum degree $\Delta \geq 3$ distinct 
from $K_{\Delta + 1}$. Generalizing Brooks' Theorem, Borodin and independently Bollob\'as and Manvel, proved 
that if $p_1, \cdots, p_s$ are non-negative integers such that $p_1 + \cdots + p_s \geq \Delta - s$, then $G$ admits a vertex partition into parts $A_1, \cdots, A_s$ such that, for $1 \leq i \leq s$,  $G[A_i]$ is $p_i$-degenerate. Here we show 
that such a partition can be performed in time $O(n+m)$. This generalizes previous results that
treated subcases of a conjecture of Abu-Khzam, Feghali and Heggernes~\cite{abu2020partitioning} which our result settles in full. 
\end{abstract}

\section{Introduction}

Brooks' Theorem is a fundamental theorem in graph coloring that draws a connection between the chromatic number and the maximum degree of a graph. 

\begin{theorem}[Brooks' Theorem \cite{brooks1941colouring}]
\emph{Every connected graph with maximum degree $\Delta \geq 3$ that is distinct from $K_{\Delta + 1}$ is $\Delta$-colorable.}
\end{theorem}  

\medskip

A graph $G$ is \emph{$d$-degenerate} if every non-empty subgraph of $G$ contains a vertex of degree at most $d$. Borodin~\cite{borodin1976decomposition} and, independently, Bollob\'as and Manvel~\cite{bollobas1979optimal} obtained the following generalization.

\begin{theorem}[Borodin~\cite{borodin1976decomposition}, Bollob\'as and Manvel~\cite{bollobas1979optimal}]\label{borodin}
Let $G$ be a non-complete connected graph with maximum degree $\Delta \geq 3$. Let $s \geq 2$ and $p_1, \cdots, p_s \geq 0$ be integers such that $\sum_{i=1}^s{ p_i} \geq \Delta - s$. Then $V(G)$ can be partitioned into sets $V_1, \cdots, V_s$ such that,  for each $i \in [1,\cdots,s]$, $G[V_i]$ is  (i) $p_i$-degenerate, and (ii) has maximum degree at most $p_i +1$. 
\end{theorem}

Brooks' Theorem follows from Theorem~\ref{borodin} by noting that a $d$-degenerate graph is $(d+1)$-colorable. We should also mention that similar generalizations and variants of Brooks' Theorem exist: see~\cite{schweser2018vertex, schweser2021partitions} for generalizations on hypergraphs, see~\cite{aboulker2021four,gol2016,bang2020digraphs_variable} for generalizations on digraphs, and see~\cite{panconesi1992improved} for a distributed version.

From an algorithmic perspective, a very short proof of Brooks' Theorem due to Lov\'asz~\cite{lovasz1975three} produces the coloring in linear time. 
The original proof of Theorem~\ref{borodin} and the alternative proof provided by 
Matamala~\cite{matamala2007vertex} are not algorithmic.
Though, another proof of Theorem~\ref{borodin} in~\cite{borodin2000variable} is algorithmic with polynomial complexity (the runtime appears to be cubic in the number of vertices).
This raises the question of whether one can possibly improve its time complexity to linear.  In view of this, several groups improved the complexity of such a partition algorithm, focusing on property (i) only.
Bonamy et al.~\cite{bonamy2017recognizing} showed that the complexity in the special case $s = 2$ with $p_1 = 0$ and $p_2 = \Delta - 2$ can be improved to quadratic for $\Delta \geq 4$ and to linear for $\Delta = 3$. Similarly, Abu-Khzam, Feghali and Heggernes~\cite{abu2020partitioning} showed that in the special case $ p_i \leq 1$ for all $i \in [s]$, it can be improved to linear. 

The object of this paper is to obtain a common generalization of these results in linear time.

\begin{theorem}\label{thm:main-bis}
There exists an algorithm that, given a non-complete connected graph $G$ with $n$ vertices, $m$ edges, and maximum degree $\Delta \geq 3$, and given a sequence  $(p_1, \cdots, p_s)$ of non-negative integers such that $s\ge 2$ and $\sum_{i=1}^s p_i \geq \Delta - s$, provides in time $O(n+m)$ a partition of $V(G)$ into sets $V_1, \cdots, V_s$ 
such that for each $i \in [s]$ $G[V_i]$ is $p_i$-degenerate.
\end{theorem}
Theorem \ref{thm:main-bis} settles a conjecture of Abu-Khzam, Feghali and Heggernes \cite{abu2020partitioning} and, in the special case $s = 2$, a problem of Bonamy et al. \cite{bonamy2017recognizing}. 

\begin{remark}
\label{remark-connected}
If the graph is not connected, we can solve the problem with the same complexity by running an algorithm for the connected case on each connected component and merging the partitions we obtain.
\end{remark}
In what follows, we always consider connected graphs.

\begin{remark}
\label{remark-greedy}
Note that the complexity of the algorithm does not depend on the length $s$ of the sequence $(p_1, \cdots, p_s)$ as in fact only the first $\Delta$ elements of this sequence will be considered by our algorithm. Indeed, those are always sufficient to fulfill the condition $\sum_{i=1}^s p_i \geq \Delta - s$.
\end{remark}

The paper is organized as follows. In Section~\ref{sec:not-regular},  we prove Theorem~\ref{thm:main-bis} in the case
when the constraint is loose (i.e. $\sum_{i=1}^s p_i > \Delta - s$) or when the graph is not $\Delta$-regular.
Then, the (more difficult) regular case with $s = 2$ is treated in Section~\ref{sec:regular}. Afterwards, in Section~\ref{section-proof-general-case}, we deduce Theorem~\ref{thm:main-bis} in full. In the final section, Section~\ref{sec:remarks}, we conclude with some remarks.

\section{The case of non-regular graphs}\label{sec:not-regular}
\label{section-non-regular}
In this section, we describe the algorithm for the non-regular case of Theorem~\ref{thm:main-bis}. The proof relies on the following folklore
observation which enables us to get a certificate of $d$-degeneracy for a graph. For completeness, we give the details. 

Given a graph $G$ and a vertex ordering $v_1,v_2,\cdots,v_n$ of $G$ we denote by $N^<(v_i)$ the neighbors of $v_i$ with lower indices, that is $N^<(v_i) = N(v_i) \cap \{v_j\ |\ j< i\}$.

\begin{observation}\label{obs-ddeg}
A graph $G$ is $d$-degenerate if and only if it admits a vertex
ordering $v_1,v_2,\cdots,v_n$ such that $|N^<(v_i)| \le d$ for every
vertex $v_i$.
\end{observation}

We call such an ordering a \emph{$d$-degenerate ordering}. 
We now describe a greedy procedure, Algorithm~\ref{algo-1}, that can handle several cases.

\begin{algorithm}
\textbf{Input}: A graph $G$, an ordered list of its vertices $v_1, \cdots, v_n$, and some integers $p_1, \cdots, p_s$\\\textbf{Output}: A partition of $V(G)$ into sets $A_1, \cdots, A_s$ so that each $G[A_i]$ is $p_i$-degenerate, or an error.

\begin{algorithmic}[1]
\State $A_1, \cdots, A_s \gets \emptyset$

\For{$v_i$ from $v_1$ to $v_n$}
  \State $k \gets 1$ \label{line-instr1}

  \While{$k\leq s$ and $|N(v_i) \cap A_k| > p_k$} \label{line-while-non-regular}
  
      \State $k++$
  \EndWhile \label{line-end-while-non-regular}
  \If{$k == s+1$}
  	\State \Return ERROR \label{line-error}
  \EndIf
  \State $A_k \gets A_k \cup v_i$ \label{line-instr2}

\EndFor   

\State \Return $(A_1, \cdots, A_s)$

\end{algorithmic}

\caption{\texttt{greedy\_partitioning}}~\label{algo-1}
\end{algorithm}

\medskip

\begin{lemma}
Let $G$ be a (not necessarily connected) graph with $n$ vertices, $m$ edges and maximum degree $\Delta \geq 3$. Let $s \geq 1$ and $p_1, \cdots, p_s \geq 0$ be integers and $P = s+\sum_{1\le i\le s} p_i$.
Given an ordering $v_1, \dots, v_n$ of $V(G)$, if $N^<(v_i) < P$ for every $1\le i\le n$, then Algorithm~\ref{algo-1}
returns, in time $O(n+m+s)$, a partition of $V(G)$ into sets $A_1, \cdots, A_s$ such that $G[A_i]$ is $p_i$-degenerate for $1 \leq i \leq s$. 
\label{lem:greedy}
\end{lemma}
\begin{remark}
Note that as  $|N^<(v_i)| \le \Delta$, the lemma applies if $\Delta < P$.
Note also that the lemma applies if $G$ is $(P-1)$-degenerate and if the ordering $v_1, \cdots, v_n$ is a $(P-1)$-degenerate ordering. 
\label{rmk:greedy-degenerate}
\end{remark}

\begin{proof}
We first prove correctness and then analyze the runtime.

\medskip

\underline{Correctness of the algorithm.} 
We first show that the algorithm does not return ERROR. Towards a contradiction, we suppose otherwise. Then for some $i \in \{1, \dots, n\}$,  $|N^<(v_i) \cap A_k| \ge p_k +1$ for all $1\le k \le s$.
Thus, $|N^<(v_i)| \ge s +\sum_{k=1}^{s} p_k = P$, which is a contradiction. So the algorithm terminates normally and returns sets $A_1, \cdots, A_s$.

It remains to show that each $G[A_k]$ is $p_k$-degenerate.
In view of Observation~\ref{obs-ddeg}, it suffices to show that the ordering $v_1,\cdots,v_n$ restricted to $A_k$ is a $p_k$-degenerate ordering of $G[A_k]$.
In other words, for every $k\in [1,\cdots,s]$ 
and every vertex $v_i\in A_k$, we have $|A_k\cap N^<(v_i)| \le p_k$.
This directly follows from the condition of the while loop in Algorithm~\ref{algo-1} \cref{line-while-non-regular}. 
  
\medskip
\underline{Runtime analysis.} 
Clearly, it suffices to show that the total cost of the while loop at \crefrange{line-while-non-regular}{line-end-while-non-regular} is $O(n+m)$.
We establish this by an amortized complexity analysis, by noting that $k$ is not incremented more than $m$ times. To see this, note that $k$ is incremented because $|N(v) \cap A_k| > p_k$ for some $k \leq s$ and $v \in V(G)$. Let $w$ be a neighbor of $v$ in $A_k$ and attribute a cost of $1$ to the edge $vw$. Clearly, the edge $vw$ is not attributed a cost more than once. Now, since the initialisation of the $A_i$'s takes time $O(s)$, the total complexity is $O(n+m+s)$.
\end{proof}
\medskip

Recall that Lemma~\ref{lem:greedy} allows us to focus on inputs such that $\Delta = s+\sum_{i=1}^s p_i$.
In the next algorithm, Algorithm~\ref{algo-non-regular}, we consider  non-$\Delta$-regular graphs $G$ with $\Delta = s+\sum_{i=1}^s p_i$. 

\begin{algorithm}

\textbf{Input}: A non-regular connected graph $G$ with maximum degree $\Delta$, and some integers $s\geq 1$ and $p_1, \cdots, p_s$ such that $p_1 + \cdots + p_s = \Delta-s$\\\textbf{Output}: A partition of $V(G)$ into sets $A_1, \cdots, A_s$ such that each $G[A_i]$ is $p_i$-degenerate

\begin{algorithmic}[1]

\State $\Delta \gets$ the maximum degree of $G$
\State $v \gets$ a vertex of $G$ of degree less than $\Delta$ \label{line-v-algo-non-regular}
\State $T \gets$ a spanning tree of $G$ rooted at $v$
\State $v_1, \cdots, v_n \gets$ an ordering of $V(G)$ obtained by a post-order traversal of $T$ starting from $v$

\State $A_1, \cdots, A_s \gets$ \texttt{greedy\_partitioning}($G$, $v_1, \cdots, v_n$, $p_1, \cdots, p_s$)

\State \Return $A_1, \cdots, A_s$

\end{algorithmic}

\caption{\texttt{non-$\Delta$-regular\_partitioning}}~\label{algo-non-regular}
\end{algorithm}

\begin{lemma}
\Cref{algo-non-regular} runs in time $O(n+m)$ and returns a partition of $V(G)$ into sets $A_1, \cdots, A_s$ such that for all $1 \leq i \leq s$, $G[A_i]$ is $p_i$-degenerate.
\label{lemma-non-regular}
\end{lemma}

\begin{proof}
We first prove correctness and then analyze the runtime. 

\medskip

\underline{Correctness of the algorithm.} 
By doing a post-order traversal of $T$ starting from $v$, the ordering $v_1, \cdots, v_n$ computed by \Cref{algo-non-regular} is a $(\Delta-1)$-degenerate ordering. In other words: 
\begin{equation*}~\label{eqn-p1}
\text{Every vertex } v_i \text{ has at most } \Delta -1 \text{ neighbors in } N^<(v_i).
\end{equation*}
Indeed, this is clear for $v_n=v$, as $|N^<(v)|= deg(v) < \Delta$. This is also clear for every vertex $v_i\neq v_n$, as its parent neighbor in $T$ does not belong to $N^<(v_i)$, hence $|N^<(v_i)|\le deg(v_i)-1 \le \Delta-1$. Then, given such an ordering, Lemma~\ref{lem:greedy} guarantees that the call of \texttt{greedy\_partitioning} returns a partition of $V(G)$ with the required properties.

\medskip
\underline{Runtime analysis.} 
Clearly, computing the maximum degree of $G$ as well as finding a vertex $v$ of degree less than $\Delta$ can be done in time $O(n+m)$. Similarly, building a spanning tree rooted at $v$ and doing a post-order traversal can be done in time $O(n + m)$. 

By Lemma~\ref{lem:greedy} the call to \texttt{greedy\_partitioning} takes  time $O(n+m+s)$. Finally, since here we have $s\le \Delta$, the running time is $O(n+m)$.
\end{proof}

\section{The case of regular graphs with $s = 2$\label{sec:regular}}
We now consider the case not handled by the previous section, that is the case where $G$ is $\Delta$-regular and where $\Delta = s+\sum_{i=1}^s{ p_i}$. We will see in \Cref{section-proof-general-case} that the case with arbitrary $s$ can be easily derived from the $s=2$ case. This is why we restrict to $s=2$ here.
So in this section we assume that we have two integers $p_A, p_B \geq 0$ such that $p_A+p_B = \Delta-2$.

Before giving the details, we give a sketch of the proof. Applying \Cref{algo-non-regular} to $G$ returns an ordering in which only vertex $v_n$ has more than $\Delta-1$ neighbors in $N^<(v_n)$ (it has exactly $\Delta$ of them). Our strategy is thus to partition $v_n$'s neighborhood more carefully in order to ease $v_n$'s coloring.

To do so, we consider a block decomposition of $G$ and select one of its end-blocks. An easy case is when this block is a ``quasi clique" (see \Cref{fig-K-}). Otherwise, we show that we can find  a vertex $z$ whose neighborhood $N(z)$ has desirable properties. We can then force the coloring of almost all vertices of $N(z)$ and call \Cref{algo-non-regular} using $z$ as  the root for the spanning tree (here $z$ plays the role of $v_n$). The most difficult part is to show that such a vertex with special neighborhood can be found in linear time.

\begin{definition}~\label{def-z-X}
For a graph $G$ with maximum degree $\Delta \geq 3$, we say that a pair $(z,X)$ formed by a vertex $z\in V(G)$ and a set $X\subseteq N(z)$ is a \emph{special neighborhood} if
\begin{itemize}
\item[a)] $|X| = \Delta - 1$,
\item[b)] $G[X]$ is not a complete graph, and
\item[c)] $G\setminus X$ is connected,
\end{itemize}
\end{definition}

Some graphs do not possess such a special neighborhood, and to deal with them we have to deal with quasi-cliques. A
\emph{quasi-clique}, denoted $K_*^-$, is the graph obtained from $K_{\Delta +1}$ by subdividing exactly one edge (see Figure~\ref{fig-K-}). Note that this graph has a degree-two vertex and that all the $\Delta +1$ remaining vertices have degree $\Delta$.
We can now present the algorithm for the case when $s=2$ and $G$ is $\Delta$-regular.

\begin{algorithm}

\textbf{Input}: A $\Delta$-regular connected graph $G$, for some $\Delta\ge 3$, distinct from $K_{\Delta+1}$, and two integers $p_A \leq p_B$ such that $p_A+p_B=\Delta-2$.\\\textbf{Output}: A partition of $V(G)$ into sets $A$ and $B$ such that $G[A]$ and $G[B]$ are $p_A$-degenerate and $p_B$-degenerate, respectively.

\begin{algorithmic}[1]

\State $A, B \gets \emptyset$

\State Perform a block-decomposition of $G$.
\State $H \gets$ an end-block of the decomposition
\State $v \gets$ the vertex linking $H$ to the rest of $G$, or any vertex if $G=H$

\If{$H$ is isomorphic to $K_*^-$} \label{line-case-quasi-clique}
	\State $A,B \gets$ \texttt{non-$\Delta$-regular\_partitioning}$(G\setminus (H \setminus \{ v\}),p_A,p_B)$ \label{line-call-algo-non-regular}
	\State Let $Y\in \{A,B\}$ be the set containing $v$ and let $Y'$ be the other set
	\State Add to $Y'$ the two neighbors of $v$ in $H$, as well as $p_{Y'}$ other vertices from $H$.
	\State Add to $Y$ the other $p_{Y}+1$ vertices of $H$.

\Else \label{line-case-no-quasi-clique}
	\State $(z, X) \gets $ \texttt{get\_special\_neighborhood}$(H, v)$ // See Algorithm~\ref{algo-z-X}
	\State $A \gets$  $p_A+2$ vertices of $X$, including two that are non-adjacent \label{line-pB-verts-to-B}
 	\State $B \gets$  the other $p_B -1$ vertices of $X$ \label{line-pA-other-verts-to-A}
	\State $T \gets $ a spanning tree of $G \setminus X$ rooted at $z$
	\State $v_{\Delta}, \cdots, v_n \gets $ ordering of $V \setminus X$ that is a post-order traversal of $T$ starting at $z$.

	\For{$v_i$ from $v_{\Delta}$ to $v_n$}

 		 \If{$|N(v_i) \cap A| \leq p_A$ } // We consider the neighborhood in $G$, not in $G \setminus X$. \label{line-test-fin-algo-delta-regular}
   
    		\State $A \gets A \cup v_i$
    
  		\Else

   			\State $B \gets B \cup v_i$ \label{line-ajout-vert-to-B}
   
   		\EndIf
   
	\EndFor

\EndIf

\State \Return $A, B$

\end{algorithmic}

\caption{\texttt{$\Delta$-regular\_bipartitioning}}~\label{algo-regular}
\end{algorithm}

\begin{theorem}
\label{thm:main-bipartition}
 \Cref{algo-regular} partitions $V(G)$ into two sets $A$ and $B$ such that $G[A]$ is $p_A$-degenerate and $G[B]$ is $p_B$-degenerate. It runs in time $O(n+m)$. 
\end{theorem}

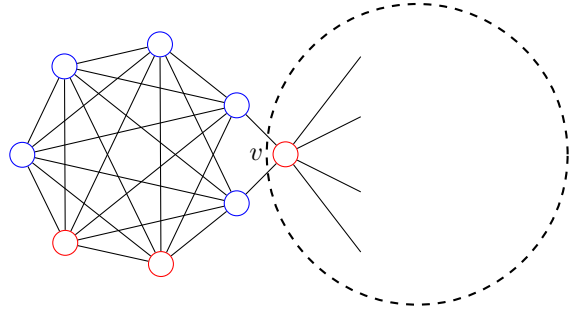
\begin{figure}
    \centering

   \begin{tikzpicture}
    
      \node[shape=circle,draw=blue] (4) at (-25.75:1.5){};
    \node[shape=circle,draw=blue] (3) at (25.75:1.5){};
    \node[shape=circle,draw=blue] (2) at (77.25:1.5){};
    \node[shape=circle,draw=blue] (1) at (128.75:1.5){};
    \node[shape=circle,draw=blue] (7) at (180.25:1.5){};
    \node[shape=circle,draw=red] (6) at (231.75:1.5){};
    \node[shape=circle,draw=red] (5) at (283.25:1.5){};
    
    \node[shape=circle,draw=red,label=west:$v$] (v) at (0:2){};
    
    \draw[black, thick, dashed] (3.75, 0) circle (2cm);

    \path [-,>=latex](1) edge node[left] {}(2);
    \path [-,>=latex](1) edge node[left] {}(3);
    \path [-,>=latex](1) edge node[left] {}(5);
    \path [-,>=latex](2) edge node[left] {}(3);
    \path [-,>=latex](2) edge node[left] {}(4);
    \path [-,>=latex](3) edge node[left] {}(v);
    \path [-,>=latex](4) edge node[left] {}(v);
    \path [-,>=latex](3) edge node[left] {}(5);
    \path [-,>=latex](2) edge node[left] {}(6);
    \path [-,>=latex](3) edge node[left] {}(7);
    \path [-,>=latex](4) edge node[left] {}(1);
    \path [-,>=latex](4) edge node[left] {}(5);
    \path [-,>=latex](4) edge node[left] {}(6);
    \path [-,>=latex](5) edge node[left] {}(7);
    \path [-,>=latex](5) edge node[left] {}(2);
    \path [-,>=latex](5) edge node[left] {}(6);
    \path [-,>=latex](6) edge node[left] {}(7);
    \path [-,>=latex](6) edge node[left] {}(1);
    \path [-,>=latex](6) edge node[left] {}(3);
    \path [-,>=latex](7) edge node[left] {}(1);
    \path [-,>=latex](7) edge node[left] {}(2);
    \path [-,>=latex](7) edge node[left] {}(4);
    
    \path [-,>=latex](v) edge node[left] {}(3, 0.5);
    \path [-,>=latex](v) edge node[left] {}(3, -0.5);
    \path [-,>=latex](v) edge node[left] {}(3, 1.3);
    \path [-,>=latex](v) edge node[left] {}(3,-1.3);

    \end{tikzpicture}
    \caption{The graph $K_*^-$ for $\Delta=6$, and some vertex partitioning for $p_A=1$ and $p_B=3$ where the vertices in $A$ (which contains $v$) are represented in red and vertices in $B$ in blue.}
    \label{fig-K-}
\end{figure}

\begin{proof}
We first prove the correctness and then analyze the runtime. 
For the properties of the algorithm \texttt{get\_special\_neighborhood}, we refer to the forthcoming \Cref{lemma-get-special-neighborhood-correction,lemma-get-special-neighborhood-complexity}.\\

\underline{Correctness of the algorithm.}
We have two subcases: either the considered end-block is isomorphic to $K_*^-$, or not.

We first consider the case when the end-block $H$ is isomorphic to $K_*^-$ with $v$ as its cut-vertex (so the condition of \cref{line-case-quasi-clique} is met).

Thanks to \texttt{non-$\Delta$-regular\_partitioning} (see Lemma~\ref{lemma-non-regular}), we can partition the vertex set of $G\setminus (V(H) \setminus \{v\})$ into sets $A$ and $B$ with the required degeneracy properties. Assume without loss of generality that $v \in A$. We extend this partial partition of $V(G)$ to the other vertices in $H$ by putting the two neighbors of $v$ in $B$, and among the $\Delta -1$ remaining vertices of our copy of $K_*^-$ we put $d_{B}$ of them in $B$ and the remaining $\Delta-1-d_B = d_A + 1$ vertices go to $A$ (see an example in Figure~\ref{fig-K-}). Note that $G[A]$ is the \textbf{disjoint} union of two $d_A$-degenerate graphs, namely $G[(A\setminus (V(H)\setminus\{v\})]$ and $G[A \cap (V(H)\setminus \{v\})] \simeq K_{d_A +1}$, hence $G[A]$ is $d_A$-degenerate. The same holds for $G[B]$, with $G[B]\setminus V(H)$ and $G[B\cap V(H)]$, which is isomorphic to the complete graph on $d_{B}+2$ vertices minus an edge, hence $G[B]$ is $d_{B}$-degenerate. \\

It remains to consider the case when $H$ is not isomorphic to $K_*^-$. Contrarily to the previous case, we will not split the graph into $H$ and $G\setminus H$. Instead, the call to \texttt{get\_special\_neighborhood} gives us a special neighborhood $(z,X)$ of $H$ such that $v \notin X$. Before proceeding with the proof, note that we postpone the proof of the existence of such a special neighborhood and of the correctness of \texttt{get\_special\_neighborhood} to \Cref{lemma-get-special-neighborhood-correction}.
Note that this special neighborhood is also a special neighborhood with respect to $G$. Points a) and b) clearly hold. For c), note that the graphs $G\setminus (V(H)\setminus \{v\})$ and $H\setminus X$ are connected, and both contain vertex $v$, therefore their union, $G\setminus X$, is also connected. 
Let us now show that  $G[A]$ is $p_A$-degenerate and $G[B]$ is $p_B$-degenerate.

Extend the ordering $v_{\Delta}, \cdots, v_{n}$, by assigning the vertices of $X$ to $v_1,\ldots, v_{\Delta -1}$. We impose a single constraint: $v_{\Delta -2}$ and $v_{\Delta -1}$ must be non-adjacent vertices of $X\cap A$.

The degeneracy of $G[A]$ follows from Observation~\ref{obs-ddeg} by considering the ordering $v_1, \cdots, v_n$ restricted to $A$. To see this, first note that the vertices of $X\cap A$
have at most $p_A$ neighbors among $\{v_1,\ldots ,v_{\Delta-2}\}\cap A$. Thus, $|A\cap N^<(v_i)| \le p_A$ for indices up to $i=\Delta-1$. For $i \geq \Delta$, the test \cref{line-test-fin-algo-delta-regular} ensures that for every vertex $v_i\in A$ we have $|A\cap N^<(v_i)| \le p_A$. Therefore, $G[A]$ is $p_A$-degenerate.

For $G[B]$ also, consider the ordering $v_1, \cdots, v_n$ restricted to $B$. The vertices of $X\cap B$ have at most $p_B -2$ neighbors in $X\cap B$. Thus, $|B\cap N^<(v_i)| \le p_B$ for indices up to $i=\Delta-1$. 
For the vertices $v_i\in B$ with $\Delta \le i< n$, we have  that $ N^<(v_i)\le \Delta +1$.
This follows from the post-order traversal considered, as in the proof of Lemma~\ref{lemma-non-regular}.
Then, as \texttt{$\Delta$-regular\_bipartitioning} adds such a vertex $v_i$ to $B$ (\cref{{line-ajout-vert-to-B}}) only if $|A\cap N^<(v_i)| \ge p_A +1$, we have
$$|B\cap N^<(v_i)| = |N^<(v_i)| - |A\cap N^<(v_i)| \le (\Delta-1) - (p_A +1) = p_B.$$
For $v_n = z$, it is different. Since at \cref{line-pB-verts-to-B} we put $p_A+2$ of its neighbors in $A$, it shall be put in $B$ and we have $$|B\cap N^<(v_n)| = |N^<(v_n)| - |A\cap N^<(v_n)| \le \Delta - (p_A +2) = p_B.$$
Therefore, $G[B]$ is $p_B$-degenerate. This completes the proof of correctness.

\medskip
\underline{Runtime analysis.}
Decomposing $G$ into blocks can be performed in linear time~\cite{tarjan1985efficient}, and testing if an end-block is isomorphic to $K_*^-$ (\cref{line-case-quasi-clique}) can be checked in time linear in the size of the end-block. Besides, we can detect in time $O(n+m)$, which branch of the {\bf if} statement to enter. In the first case ($H$ is isomorphic to $K_*^-$) identifying $v$ and running \texttt{non-$\Delta$-regular\_partition} takes time $O(n+m)$ (by \Cref{lemma-non-regular}). Splitting the other vertices of $U$ into $A$ and $B$ can be done in the same complexity, hence the overall complexity for this subcase is $O(n+m)$. In the second case, finding a special neighborhood takes time $O(n+m)$ by \Cref{lemma-get-special-neighborhood-complexity}. Then, similarly as in Section 2, the complexity of the remaining instructions is clearly $O(n+m)$, which concludes the proof.

\end{proof}


\begin{figure}
    \centering

   \begin{tikzpicture}
  
    \node[shape=circle,draw=black,label=west:$v$,minimum size=0.5mm] (v) at (0,0){};
    \draw (0,-2.5) node[] {$L_0$};
    \draw (1,0) ellipse (0.4cm and 2cm);
    \draw (1,-2.5) node[] {$L_1$};
    \draw (2,0) ellipse (0.4cm and 2cm);
    \draw (2,-2.5) node[] {$L_2$};
    \draw (3,0) node[] {\Huge{$\ldots$}};
    \draw (3,-2.5) node[] {{$\ldots$}};
    \draw (4,0) ellipse (0.4cm and 2cm);
    \draw (4,-2.5) node[] {$L_{\ell-1}$};
    \draw (5,0) ellipse (0.4cm and 2cm);
    \draw (5,-2.5) node[] {$L_{\ell}$};

    \draw[black, dashed] (4.5,0) ellipse (0.8cm and 0.5cm);
    \draw (4,-0.6) node[] {$X$};
    
    
    \path [-,>=latex,thick](v) edge node[left] {}(0.8,1.5);
     \path [-,>=latex,thick](v) edge node[left] {}(0.8,0);
     \path [-,>=latex,thick](v) edge node[left] {}(0.8,-1.5);
     
     \path [-,>=latex,thick](1.2,1.5) edge node[left] {}(1.8,1.2);
     \path [-,>=latex,thick,thick](1.2,-1.5) edge node[left] {}(1.8,-1.2);
     
     \path [-,>=latex,thick,thick](1.2,0.3) edge node[left] {}(1.8,-0.3);
     
      \path [-,>=latex,thick](4.2,1.4) edge node[left] {}(4.8,1.5);
     \path [-,>=latex,thick](4.2,-1.4) edge node[left] {}(4.8,-1.5);
     
     
     \node[shape=circle,draw=black,label=north:$z$] (z) at (5,0.9){};
     \path [-,>=latex,thick](z) edge node[left] {}(5,0);
     \path [-,>=latex,thick](z) edge node[left] {}(4,0);
     \path [-,>=latex,thick](z) edge node[left] {}(4,0.5);

    \node[shape=circle,draw=black,label=north:$u$] (u) at (5,-1){};
    
    \path [-,>=latex,thick,thick](u) edge node[left] {}(4,-1);
    \path [-,>=latex,thick,thick](u) edge node[left] {}(4,-0.2);
    
    \end{tikzpicture}
    \caption{The layers $L_0,L_1,\cdots, L_{\ell-1},L_\ell$ of $H$.}
    \label{fig_couches}
\end{figure}
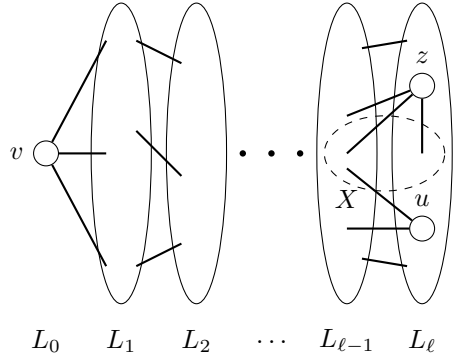

\begin{algorithm}

\textbf{Input}: A 2-connected graph $H$ distinct from $K^-_*$ and from $K_{\Delta+1}$, and a vertex $v \in H$. All vertices except possibly $v$ have degree $\Delta$ in $H$.\\\textbf{Output}: A special neighborhood $(z, X)$ of $H$ such that $v \notin X$.
\caption{\texttt{get\_special\_neighborhood}}~\label{algo-z-X}
\begin{algorithmic}[1]

\State Perform a BFS on $H$ starting from $v$
\State Partition the vertices into $L_0 = \{v\}, L_1, \cdots, L_\ell$ : $L_i$ contains the vertices at distance $i$ from $v$
\State $k \gets $min($|N(z) \cap L_{\ell-1}|$) for each $z\in L_\ell$ \label{line-k-equals-min}

\If{$k \geq 3$}
\State $z \gets$ a vertex in $L_\ell$ with exactly $k$ neighbors in $L_{\ell-1}$
\State $x_1, x_2 \gets$ two non-adjacent neighbors of $z$.
\State $x_3 \gets$ a neighbor in $N(z) \cap L_{\ell-1}$ distinct from $x_1$ and $x_2$
\State $X \gets N(z) \setminus x_3$
\State \Return $(z,X)$ \label{line-return-k3}

\EndIf\\

\For{each \textbf{non-marked} $z \in L_\ell$  with exactly $k$ neighbors in $L_{\ell-1}$} \label{line-non-marked-for}
	\If{$\exists$ a non-edge $x_1,x_2$ in $ N(z)$ such that $|\{x_1, x_2\} \cap L_{\ell-1}| < k$} \label{line-test-easy-loop}
		\State $x_3 \gets$ a vertex in $(N(z)\cap L_{\ell-1}) \setminus \{x_1, x_2\}$ \label{line-select-easy-loop}
		\State $X \gets N(z) \setminus x_3$
		\State \Return $(z,X)$ \label{line-return-easy-case}
	
	\Else
		\State Mark $z$ as well as all its neighbors in $L_\ell$. \label{line-marking}
	\EndIf
\EndFor \label{line-end-for-non-marked}\\

\If{$k==1$}
\State $u \gets$ a vertex of $L_\ell$ with exactly one neighbor in $L_{\ell-1}$ \label{line-case-k1}
\State $C \gets H[N[u]\cap L_\ell]$
\State $x_1 \gets$ a vertex of $L_{\ell-1} \cap N(C)$ with at most $|C|/2$ neighbors in $C$
\State $z \gets$ a neighbor of $x_1$ in $C$\label{line-build-z-x1}
\State $x_2, x_3 \gets$ two vertices $C \setminus N(x_1)$
\State $X \gets N(z) \setminus x_3$ \label{line-build-X-k1}
\State \Return $(z, X)$ \label{line-return-k1}\\

\Else $\;$// Necessarily $k=2$
\State $u \gets $ a vertex of $L_\ell$ having exactly two neighbors in $L_{\ell-1}$ \label{line-case-k2}
\State $x_1, x_2 \gets$ the neighbors of $u$ in $L_{\ell-1}$
\State $y_1 \gets $ the neighbor of $x_1$ in $L_{\ell-2}$
\State $y_2 \gets $ the neighbor of $x_2$ in $L_{\ell-2}$
\If{$y_1 == v$}
	\State Exchange $x_1$ and $x_2$, also $y_1$ and $y_2$
\EndIf

\State $X \gets N(x_1) \setminus \{u\}$
\State \Return $(x_1, X)$\label{line-end-algo-special}

\EndIf

\end{algorithmic}

\end{algorithm}

\begin{lemma}~\label{lem-z-X}
Algorithm~\ref{algo-z-X} returns a special neighborhood $(z,X)$ of $H$, such that $v\notin X$.
\label{lemma-get-special-neighborhood-correction}
\end{lemma}

\begin{proof}
    We must show that the returned pair has properties a), b) and c) of \Cref{def-z-X}. Since $H$ is different from $K_{\Delta + 1}$ and all vertices of $H\setminus\{v\}$ have degree $\Delta$, the neighborhood of every vertex $z\neq v$ contains a non-edge. So finding some $(z,X)$ satisfying properties a) and b) of \Cref{def-z-X} is easy. To find a special neighborhood, the difficulty thus lies in guaranteeing the connectivity of $H\setminus X$.

    We begin by partitioning $V(H)$ into sets $L_0,L_1, \dots, L_{\ell}$  so that a vertex belongs to $L_i$ if it is at distance $i$ from $v$. Since $H$ is not $K_{\Delta+1}$ and since vertices in $L_1$ have degree $\Delta$, at least one vertex in $L_1$ has a neighbor in $L_2$, hence $\ell \geq 2$. By the execution of \cref{line-k-equals-min}, $k$ is the minimum number of neighbors a vertex $z$ has in $L_{\ell-1}$ over all $z \in L_\ell$. We split the proof into three subcases with respect to the value of $k$.

\underline{Case $k \geq 3$}.
Since the graph contains no $K_{\Delta+1}$, there is a non-edge in $H[N(z)]$, say $x_1x_2$. Moreover, as $k\geq 3$, there exists a vertex $x_3\in N(z)\cap L_{\ell-1}$ distinct from $x_1$ and $x_2$. Setting $X$ to $N(z)\setminus \{x_3\}$ ensures a) and b) as $X$ contains $x_1$ and $x_2$.  Since the vertices of $X$ belong to $L_{\ell-1} \cup L_\ell$, we satisfy the condition that $v\notin X$.

It remains to show condition c), that is the fact that $H \setminus X$ is connected. It suffices to show that any vertex in $H \setminus X$ is connected to $v$. This follows by induction from $L_0$ to $L_{\ell}$. Indeed, every vertex in $L_i$ has a neighbor in $L_{i-1}\setminus X$. For $i<\ell$, this holds because
actually $L_{i-1}\setminus X = L_{i-1}$, as $X\subseteq L_{\ell-1}\cup L_\ell$.
For $i=\ell$, any vertex in $L_\ell$ has at least $k$ neighbors in $L_{\ell-1}$. As $X$ contains only $k-1$ vertices of $L_{\ell-1}$, any vertex in $L_\ell$ has a neighbor in $L_{\ell-1}\setminus X$ hence $L_\ell \setminus X$ is also connected to $v$.

\medskip

\underline{Case $k = 1$}. In the easiest case, $(z,X)$ is returned from the first loop (\cref{line-return-easy-case}), in which case $z$ has one neighbor in $L_{\ell-1}$, the set $X$ has size $\Delta-1$ and $H[X]$ is not complete. As such, conditions a) and b) are verified. Moreover, $X$ contains only vertices from $L_\ell$ (the furthest layer from $v$). Proving that $H \setminus X$ is connected is similar to the case when $k\geq 3$: $L_1$ to $L_{\ell-1}$ are included in $v$'s connected component and every vertex of $L_\ell$ has a neighbor in $L_{\ell-1}$.
We now claim that not returning at \cref{line-return-easy-case} and entering \cref{line-case-k1} implies the following: for every vertex $z\in L_\ell$ having only one neighbor in $L_{\ell-1}$ the graph $H[N[z]\cap L_\ell]$ is isomorphic to $K_{\Delta}$. This would be immediate if we checked every vertex instead of every non-marked vertex, see \cref{line-non-marked-for}. So we have to show that no vertex marked \cref{line-marking} can pass the the test \cref{line-test-easy-loop}. Indeed, if some $z$ fails this test (\cref{line-test-easy-loop}) it means that $H[z]\cap L_\ell$ is a clique of order $\Delta$. No vertex in this clique has a non-edge in its neighborhood restricted to $L_\ell$, hence all vertices of this clique can safely be marked and not examined later.

We pick some vertex $u \in L_\ell$ with only one neighbor in $L_{\ell-1}$ and define $C = H[N[u] \cap L_\ell]$. Recall that $C$ is a clique of order $\Delta$. We select a vertex $x_1 \in L_{\ell-1} \cap N(C)$ which is connected to at most half the vertices of $C$ (see Figure~\ref{fig:k1}). This is possible because $C$ has at least two neighbors in $L_{\ell-1}$ for otherwise $C\cup N(C)$ would be a clique of order $\Delta+1$. As $C$ has order $\Delta \ge 3$, $x_1$ has at least 
two non-neighbors in $C$ that we denote $x_2$ and $x_3$.
At \cref{line-build-X-k1} we define $X$ as $N(z) \setminus x_3$.
Let us prove that this set, which we return at \cref{line-return-k1}, has the desired properties. Property a) is true by construction. Property b) also holds because $x_2$ was chosen among the non-neighbors of $x_1$ so $H[X]$ contains a non-edge. Finally, since $X \subseteq L_{\ell-1} \cup L_\ell$, it follows that $v\notin X$. 
It remains to show that $H\setminus X$ is connected.

As for the case $k\geq 3$, the connected component of $H\setminus X$ containing vertex $v$ contains all the vertices of 
$L_1,\cdots, L_{\ell-1}\setminus X = L_{\ell-1}\setminus\{x_1\}$. Besides, since $x_3$ is not connected to $x_1$ but must have a neighbor in $L_{\ell-1}$, it is connected to $v$. Therefore $z$ is connected to $v$ through $x_3$. It remains to show that the vertices of $L_\ell \setminus (X\cup \{z, x_3\}) = L_\ell \setminus V(C)$ belong to $v$'s connected component. Let $w$ be a vertex in this set.
Since $H$ is 2-connected, there exist two vertex-disjoint paths $P_1, P_2$ from $w$ to $v$. We can assume then that for instance $P_2$ does not contain $x_1$. Since $C$ is a connected component of $H[L_\ell]$, $P_2$ reaches $L_{\ell-1}\setminus \{x_1\}$ before possibly reaching $C$. So this part of $P_2$ avoids $X$ and ensures $w$ to be in $v$'s connected component. This completes the case.

\begin{figure}
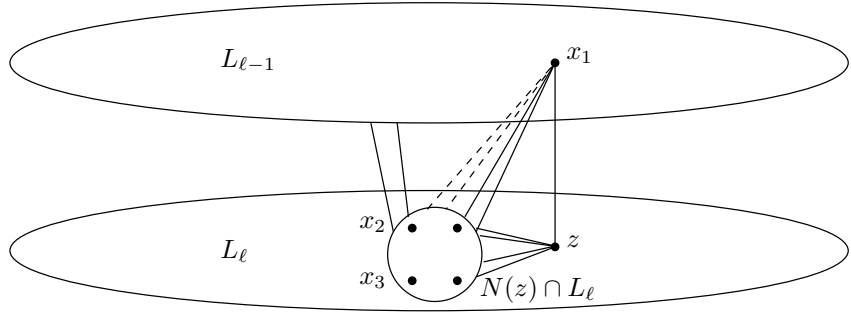

  \centering
  \ifx\XFigwidth\undefined\dimen1=0pt\else\dimen1\XFigwidth\fi
\divide\dimen1 by 10064
\ifx\XFigheight\undefined\dimen3=0pt\else\dimen3\XFigheight\fi
\divide\dimen3 by 3718
\ifdim\dimen1=0pt\ifdim\dimen3=0pt\dimen1=2071sp\dimen3\dimen1
  \else\dimen1\dimen3\fi\else\ifdim\dimen3=0pt\dimen3\dimen1\fi\fi
\tikzpicture[x=+\dimen1, y=+\dimen3]
{\ifx\XFigu\undefined\catcode`\@11
\def\temp{\alloc@1\dimen\dimendef\insc@unt}\temp\XFigu\catcode`\@12\fi}
\XFigu2071sp
\ifdim\XFigu<0pt\XFigu-\XFigu\fi
\clip(2235,-7529) rectangle (12299,-3811);
\tikzset{inner sep=+0pt, outer sep=+0pt}
\pgfsetlinewidth{+15\XFigu}
\draw  (7267,-6795) ellipse [x radius=+5017,y radius=+720];
\draw  (7267,-4545) ellipse [x radius=+5017,y radius=+720];
\filldraw  (7065,-6525) circle [radius=+45];
\filldraw  (8775,-6750) circle [radius=+45];
\filldraw  (8775,-4545) circle [radius=+45];
\draw  (7335,-6840) circle [radius=+564];
\filldraw  (7065,-7155) circle [radius=+45];
\filldraw  (7605,-7155) circle [radius=+45];
\filldraw  (7605,-6525) circle [radius=+45];
\draw (8775,-4545)--(7830,-6570);
\draw (8775,-4545)--(7695,-6390);
\draw (7830,-6525)--(8820,-6750);
\draw (7875,-6615)--(8820,-6750);
\draw (8730,-6750)--(7920,-6930);
\draw (8775,-6750)--(7830,-7110);
\pgfsetdash{{+90\XFigu}{+90\XFigu}}{++0pt}
\draw (7245,-6300)--(8775,-4545);
\draw (8775,-4545)--(7470,-6300);
\pgfsetdash{}{+0pt}
\draw (8775,-4545)--(8775,-6750);
\draw (7020,-6390)--(6885,-5265);
\draw (6840,-6570)--(6570,-5265);
\pgftext[base,left,at=\pgfqpointxy{6435}{-6525}] {\fontsize{10}{12}\usefont{T1}{ptm}{m}{n}$x_2$}
\pgftext[base,left,at=\pgfqpointxy{6435}{-7200}] {\fontsize{10}{12}\usefont{T1}{ptm}{m}{n}$x_3$}
\pgftext[base,left,at=\pgfqpointxy{7875}{-7335}] {\fontsize{10}{12}\usefont{T1}{ptm}{m}{n}$N(z)\cap L_\ell$}
\pgftext[base,left,at=\pgfqpointxy{8910}{-6750}] {\fontsize{10}{12}\usefont{T1}{ptm}{m}{n}$z$}
\pgftext[base,left,at=\pgfqpointxy{8910}{-4500}] {\fontsize{10}{12}\usefont{T1}{ptm}{m}{n}$x_1$}
\pgftext[base,left,at=\pgfqpointxy{4770}{-6885}] {\fontsize{10}{12}\usefont{T1}{ptm}{m}{n}$L_\ell$}
\pgftext[base,left,at=\pgfqpointxy{4770}{-4590}] {\fontsize{10}{12}\usefont{T1}{ptm}{m}{n}$L_{\ell-1}$}
\endtikzpicture%
    \caption{The case $k=1$ with $\Delta = 5$.}
    \label{fig:k1}
\end{figure}

\medskip

\underline{Case $k=2$}. As in the case $k=1$, $(z,X)$ can be returned by the first loop (\cref{line-return-easy-case}). In that case, conditions a) and b) and c) hold for the same reasons as in the case $k=2$.

We are now at \cref{line-case-k2} and consider a vertex $u\in L_\ell$ having only two neighbors in $L_{\ell-1}$ that we denote $x_1$ and $x_2$ . We claim that not returning at \cref{line-return-easy-case} and executing \cref{line-case-k2} implies the following: $H[N[u]\setminus \{x_1\}]$ and $H[N[u]\setminus \{x_2\}]$ are isomorphic to $K_{\Delta}$ (see Figure~\ref{fig:k2}). As $x_1$ and $x_2$ have $\Delta -1$ neighbors in $L_\ell$, both of them necessarily have their $\Delta^\text{th}$ neighbor in $L_{\ell-2}$. We denote them by $y_1$ and $y_2$. As $H$ is distinct from $K_*^-$, 
these vertices are distinct. Therefore, we can choose $x_1$ and $y_1$ be such that $y_1 \neq v$ (exchanging the $x_i$'s and $y_i$'s if necessary).

Consider the set $X = N(x_1) \setminus\{u\}$. We claim that $(x_1, X)$ is a special neighborhood, which will complete the proof. Clearly, $(x_1, X)$ satisfies a) and since $y_1\in L_{\ell-2}$ has no neighbors in $L_\ell$, it also satisfies b). Furthermore, since $v\neq y_1$ and $X\setminus \{y_1\} \subseteq L_\ell$, we have that $v\notin X$.
To see that c) holds, observe as before that the connected component of $H\setminus X$ containing $v$ 
contains all the vertices of $L_1,\cdots, L_{\ell-2}\setminus \{y_1\}$. Then, consider the connected components of $H[L_{\ell-1} \cup L_{\ell}]$. Given the maximum degree $\Delta$, one of them corresponds to $H[N[u]]$. As $N[u]\setminus X = \{u,x_1,x_2\}$, the path $x_1ux_2y_2$ ensures that the vertices of $N[u]\setminus X$ are in $v$'s connected component. Let any other connected component $Q$ of 
$H[L_{\ell-1} \cup L_{\ell}]$. By what precedes, $Q\setminus X =Q$. By the 2-connectivity of $H$, $N(Q)\setminus \{y_1\}$ is non-empty, which connects $Q$ to $v$.
This completes the case and therefore the proof of the lemma.

\begin{figure}
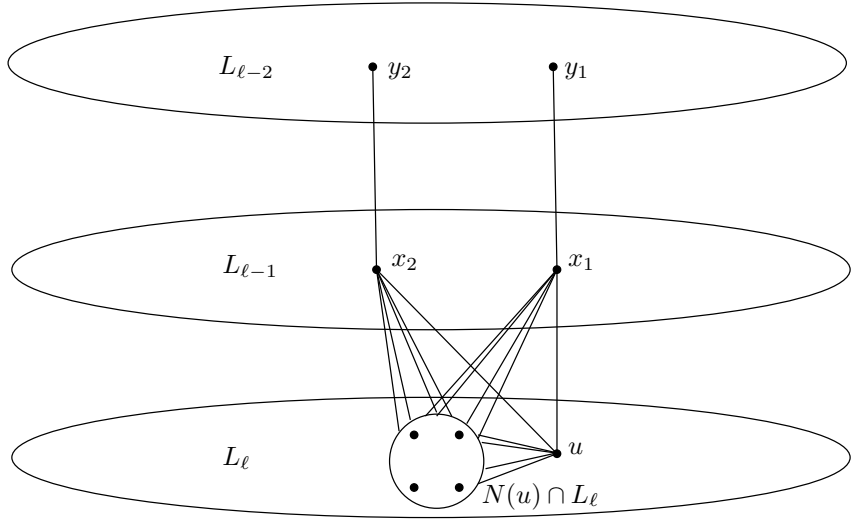

  \centering

\ifx\XFigwidth\undefined\dimen1=0pt\else\dimen1\XFigwidth\fi
\divide\dimen1 by 10109
\ifx\XFigheight\undefined\dimen3=0pt\else\dimen3\XFigheight\fi
\divide\dimen3 by 6193
\ifdim\dimen1=0pt\ifdim\dimen3=0pt\dimen1=2071sp\dimen3\dimen1
  \else\dimen1\dimen3\fi\else\ifdim\dimen3=0pt\dimen3\dimen1\fi\fi
\tikzpicture[x=+\dimen1, y=+\dimen3]
{\ifx\XFigu\undefined\catcode`\@11
\def\temp{\alloc@1\dimen\dimendef\insc@unt}\temp\XFigu\catcode`\@12\fi}
\XFigu2071sp
\ifdim\XFigu<0pt\XFigu-\XFigu\fi
\clip(2190,-7529) rectangle (12299,-1336);
\tikzset{inner sep=+0pt, outer sep=+0pt}
\pgfsetlinewidth{+15\XFigu}
\draw  (7267,-6795) ellipse [x radius=+5017,y radius=+720];
\draw  (7267,-4545) ellipse [x radius=+5017,y radius=+720];
\filldraw  (7065,-6525) circle [radius=+45];
\filldraw  (8775,-6750) circle [radius=+45];
\filldraw  (8775,-4545) circle [radius=+45];
\draw  (7335,-6840) circle [radius=+564];
\filldraw  (7065,-7155) circle [radius=+45];
\filldraw  (7605,-7155) circle [radius=+45];
\filldraw  (7605,-6525) circle [radius=+45];
\filldraw  (6615,-4545) circle [radius=+45];
\draw  (7222,-2070) ellipse [x radius=+5017,y radius=+720];
\filldraw  (6570,-2115) circle [radius=+45];
\filldraw  (8730,-2115) circle [radius=+45];
\draw (8775,-4545)--(7830,-6570);
\draw (8775,-4545)--(7695,-6390);
\draw (7830,-6525)--(8820,-6750);
\draw (7875,-6615)--(8820,-6750);
\draw (8730,-6750)--(7920,-6930);
\draw (8775,-6750)--(7830,-7110);
\pgfsetdash{}{+0pt}
\draw (8775,-4545)--(8775,-6750);
\draw (8775,-4545)--(7335,-6300);
\draw (8775,-4545)--(7200,-6300);
\draw (6615,-4545)--(7515,-6300);
\draw (6615,-4545)--(7335,-6255);
\draw (6615,-4545)--(7020,-6345);
\draw (6615,-4545)--(6885,-6480);
\draw (6615,-4545)--(8775,-6750);
\draw (6615,-4545)--(6570,-2115);
\draw (8775,-4545)--(8730,-2115);
\pgftext[base,left,at=\pgfqpointxy{7875}{-7335}] {\fontsize{10}{12}\usefont{T1}{ptm}{m}{n}$N(u)\cap L_\ell$}
\pgftext[base,left,at=\pgfqpointxy{8910}{-6750}] {\fontsize{10}{12}\usefont{T1}{ptm}{m}{n}$u$}
\pgftext[base,left,at=\pgfqpointxy{8910}{-4500}] {\fontsize{10}{12}\usefont{T1}{ptm}{m}{n}$x_1$}
\pgftext[base,left,at=\pgfqpointxy{4770}{-6885}] {\fontsize{10}{12}\usefont{T1}{ptm}{m}{n}$L_\ell$}
\pgftext[base,left,at=\pgfqpointxy{4770}{-4590}] {\fontsize{10}{12}\usefont{T1}{ptm}{m}{n}$L_{\ell-1}$}
\pgftext[base,left,at=\pgfqpointxy{6795}{-4500}] {\fontsize{10}{12}\usefont{T1}{ptm}{m}{n}$x_2$}
\pgftext[base,left,at=\pgfqpointxy{6750}{-2205}] {\fontsize{10}{12}\usefont{T1}{ptm}{m}{n}$y_2$}
\pgftext[base,left,at=\pgfqpointxy{8865}{-2205}] {\fontsize{10}{12}\usefont{T1}{ptm}{m}{n}$y_1$}
\pgftext[base,left,at=\pgfqpointxy{4725}{-2205}] {\fontsize{10}{12}\usefont{T1}{ptm}{m}{n}$L_{\ell-2}$}
\endtikzpicture%

    \caption{The case $k=2$ with $\Delta = 6$.}
    \label{fig:k2}
\end{figure}

\end{proof}

\begin{lemma}
   Algorithm~\ref{algo-z-X} runs in time $O(n+m)$.
\label{lemma-get-special-neighborhood-complexity}
\end{lemma}

\begin{proof}
  The first step of the algorithm, performing a BFS on $H$ and partitioning the vertices according to their distance from $v$, can be done in time $O(n+m)$. In order to speed up some later operations to achieve the desired complexity, we precompute an array $distToV$ such that for each $u\in H, distToV[u]$ contains the distance from $v$ to $u$. Then, computing $k$ can be done in time $O(n+m)$ by looping over each $z\in L_\ell$, each time visiting the neighbors of $z$ and checking with the $distToV$ array whether they belong to $L_{\ell-1}$. We also initialize here a boolean array $isNeighbWithZ$ to false for each vertex of $H$.
  
The case $k \geq 3$ takes time $O(n+m)$. Indeed, in time $O(\Delta)$ we first update $isNeighbWithZ$ to assign true for each neighbor of $z$. Thanks to this, the selection of $x_1,x_2$ can be done in $O(|E(H(N[z]))|) = O(\Delta+m)$.

The for loop at \crefrange{line-non-marked-for}{line-end-for-non-marked} can be done in time $O(n+m)$. First, the list of vertices in $L_\ell$ with $k$ neighbors in $L_{\ell-1}$ can be made in time $O(n+m)$ thanks to $distToV$.
Second, note that the sets $N(z)\cap L_\ell$ for each $z$ we consider at \cref{line-non-marked-for} are disjoint. This comes from the fact that we mark all vertices in $N(z)\cap L_\ell$ for each $z$ we investigate.

We now only need to observe that \crefrange{line-select-easy-loop}{line-return-easy-case}  (executed at most once) take time $O(n+m)$ and that \cref{line-test-easy-loop} takes time
$O(|E(H[N(z)]|)$ per iteration. Over all iterations on unmarked vertices,
this sums up to $O(|E(H)| + 2\Delta |L_\ell|) = O(m)$ as the sets $N(z)\cap L_\ell$ are disjoint. The $2\Delta$ term comes from the fact that at $k\leq 2$, hence $H[N(z)]$ is made of $H[N(z)\cap L_\ell]$ plus at most two other vertices.
The marking of the vertices takes time $O(m)$ in total.

The body of the ``if $k==1$'' part takes time $O(n+m)$. The only non-trivial point is the selection of $x_1$. To achieve this complexity, we can first build and fill a boolean array $isInC$. We can select $x_1$ by first generating the list of neighbors of $C$ which are in $L_{\ell-1}$. Then for each such vertex we count how many of its neighbors are in $C$. This consists in visiting disjoint edges, hence it can be done in time $O(n+m)$.
  
Finally, it is clear that the body of the ``else'' part ($k=2$) can also be done in time $O(n+m)$.
\end{proof}

\section{The proof of Theorem~\ref{thm:main-bis} (general case)}
\label{section-proof-general-case}
Consider Algorithm~\ref{algo-main}, and let us verify that it fulfills Theorem~\ref{thm:main-bis}.

\begin{algorithm}

\textbf{Input}: A graph $G$ with maximum degree $\Delta$, and a sequence of length $s\ge 2$ of non-negative integers $p_1, \cdots, p_s$ such that $\sum_{i=1}^s p_i \geq \Delta - s$\\\textbf{Output}: A partition of $V(G)$ into sets $A_1, \cdots, A_s$ so that each $G[A_i]$ is $p_i$-degenerate

\begin{algorithmic}[1]
\State $s' \gets \min(s,\Delta)$
\If{$\sum_{i=1}^{s'} p_i > \Delta - s'$}
    \State $v_1,\cdots,v_{n} \gets$ Any ordering of $V(G)$
  	\State \Return \texttt{greedy\_partitioning}$(G,v_1,\cdots,v_n,p_1,\cdots,p_{s'})$
\EndIf $\ \ \ \ \ \ \ \ $// Now $\sum_{i=1}^{s'} p_i = \Delta - s'$
\If{$G$ is not regular}
    \State \Return \texttt{non-$\Delta$-regular\_partitioning}$(G,s,p_1,\cdots,p_{s'})$
\EndIf $\ \ \ \ \ \ \ \ $// Now $G$ is $\Delta$-regular \label{linefirsttwocasefinal}
\State $p^- \gets \min(p_1,p_2)$
\State $p^+ \gets \max(p_1,p_2)$
\State $A,B \gets$ \texttt{$\Delta$-regular\_bipartitioning}$(G,p^-,\Delta-2-p^-)$ \label{call-regular-final}
\State $v_1,\cdots,v_{|B|} \gets$ a $(\Delta-2-p^-)$-degenerate ordering of $G[B]$
\State $A^+,A_3,\cdots,A_{s'} \gets$ \texttt{greedy\_partitioning}$(G[B],v_1\cdots,v_{|B|},p^+,p_3,\cdots,p_{s'})$
\If{$p_1\le p_2$}
    \State \Return $A,A^+,A_3,\cdots,A_{s'}$
\Else
    \State \Return $A^+,A,A_3,\cdots,A_{s'}$
\EndIf    
\end{algorithmic}

\caption{\texttt{main\_algorithm}}~\label{algo-main}
\end{algorithm}

\medskip

The trick we use here, to have a complexity independent of $s$, the number of $p_i$'s in input, is to restrict to the first $\Delta$ of the $p_i$'s (Algorithm~\ref{algo-main} might not read the whole input), and output only $\min(s,\Delta)$ sets.

\begin{proof}[Proof of \Cref{thm:main-bis}]
We first prove the correctness and then analyze the runtime. 

\medskip
\underline{Correctness of the algorithm.}
If $s$ is greater than $\Delta$ we can ignore all the $p_i$ for $i > \Delta$. Indeed, in that case the $\Delta$ first $p_i$'s, even if they are all equal to zero, sum up to at least $0\ge \Delta - s$. Thus, the correctness of the algorithm for the case $s\le \Delta$, implies its correctness in full. As in the algorithm, let us consider
the following three cases: Either $\sum_{i=1}^{s'}{p_i} > \Delta-s'$, $\sum_{i=1}^{s'}{p_i} = \Delta-s'$ and $G$ is not $\Delta$-regular, or $\sum_{i=1}^{s'}{p_i} = \Delta-s'$ 
and $G$ is $\Delta$-regular.
The first two cases are handled by the calls to 
\texttt{greedy\_partitioning} (by \Cref{lem:greedy}), and to
\texttt{non-$\Delta$-regular\_partitioning} (by \Cref{lemma-non-regular}), respectively.
For the third case, we only prove the case $p_1\le p_2$, the case $p_1> p_2$ being similar.
This third case is handled in two steps, first by a call to \texttt{$\Delta$-regular\_bipartitioning}
partitioning $V(G)$ into a $p_1$-degenerate graph, $G[A]$, and a $(\Delta - 2 - p_1)$-degenerate graph, $G[B]$ 
(by \Cref{thm:main-bipartition} since $p_1\le p_2 \le \Delta -2-p_1$), and then by a call to 
\texttt{greedy\_partitioning} refining $B$ into $p_2$-, $\ldots, p_{s'}$-degenerate graphs 
(by \Cref{rmk:greedy-degenerate} since $\Delta -2-p_1 < (s'-1) + \sum_{i=2}^{s'} p_i$).

\medskip
\underline{Runtime analysis.}
The time complexity of \Cref{algo-main} lies in the calls to other algorithms and in the instructions
within the algorithm, the latter taking clearly only time $O(n+m+s')$. 
By \Cref{lem:greedy}, \Cref{lemma-non-regular}, and \Cref{thm:main-bipartition},
these calls take time $O(n+m+s')$ and $O(n+m)$. As $s'\le \Delta \le n$, the overall complexity is $O(n+m)$.
\end{proof}

\section{Final remarks}\label{sec:remarks}

\subsection{Bounding the maximum degree}\label{sec:degree}

In~\cite{borodin2000variable} the authors describe an algorithm turning a partition fulfilling point (i) of Theorem~\ref{borodin} into one fulfilling both (i) and (ii). Their algorithm is a succession of individual vertex moves, that is, replacing two sets $V_i,V_j$ 
with $V_i\setminus \{u\}$ and $V_j\cup\{u\}$. Each such step
diminish an energy-like function whose image lies in $[-4m,2m]$, so we can safely bound the number of moves by $6m$. At each step, updating the partition and maintaining the list of vertices violating point (ii) needs $O(deg(u))= O(\Delta)$ time. So the whole algorithm runs in $O(\Delta m)$ time. 

\subsection{Perspectives}
Theorem~\ref{borodin} has been generalized in~\cite{borodin2000variable} by replacing the notion of degeneracy by the notion of \emph{variable degeneracy}. This improvement was in turn recently generalized in the context of digraphs~\cite{bang2020digraphs_variable}.  
This generalization is achievable in polynomial time (quadratic or less), but it seems difficult to perform it in linear time, as it relies on finding cycles with particular properties. It would be interesting to have an algorithm performing such a partition in linear time.

\bigskip

\noindent
\textbf{Acknowledgements.} Supported by Agence Nationale de la Recherche (France) under
research grant ANR DIGRAPHS ANR-19-CE48-0013-01, and ANR GATO ANR-16-CE40-0009. The authors are grateful to the organizers of JGA 2021 for making this collaboration possible. We also thank the referees for their feedback which prompted improvements in the complexity of our algorithm (by removing a multiplicative factor $s$). They also helped us write this paper in a clearer way.

\end{document}